\documentclass[11pt,leqno]{article} 
\usepackage{amsmath, amsthm, amsfonts, amssymb, color,
  enumerate,mathrsfs,authblk,verbatim}
\allowdisplaybreaks

\theoremstyle{definition} 
\theoremstyle{remark}

\newcommand{\scr}[1]{\mathscr #1} \definecolor{wco}{rgb}{0.5,0.2,0.3}

\newcommand{\be}{\begin{eqnarray}} \newcommand{\ee}{\end{eqnarray}}
\newcommand{\ce}{\begin{eqnarray*}} \newcommand{\de}{\end{eqnarray*}}
\newtheorem{theorem}{Theorem}[section]
\newtheorem{lemma}[theorem]{Lemma}
\newtheorem{remark}[theorem]{Remark}

\newtheorem{proposition}[theorem]{Proposition}

\newtheorem{corollary}[theorem]{Corollary}

 \def\DD{\Delta} 
 \def\HS{\text{\tiny\rm HS}}
  \def\e{\mathrm{e}} 
 \def\a{\alpha}  
   
 \def\[{{\Big[}} \def\]{{\Big]}}
\def\<{{\langle}} \def\>{{\rangle}} \def\({{\Big(}} \def\){{\Big)}}

 \def\E{\mathbb E}

\newcommand\1{\hbox{\kern.375em\vrule height1.57ex depth-.1ex
    width.05em\kern-.375em \rm 1}}

\def\sD{{\mathscr D}}

\def\geq{\geqslant} \def\leq{\leqslant} \def\ge{\geqslant}
\def\le{\leqslant}

\def\si{\sigma} \def\R{\mathbb R}\def\nn{\nabla}
\def\ff{\frac} \def\ss{\sqrt} \def\B{\mathscr B}
  \def\dd{\delta} \def\vv{\varepsilon}
\def\<{\langle} \def\>{\rangle} 
 \def\nn{\nabla} \def\pp{\partial} 
\def\d{\text{\rm{d}}}  \def\D{\scr D}
\def\si{\sigma}  \def\e{\text{\rm{e}}} 
\def\P{\mathbb P}  
   \def\ll{\lambda}

\def\Ric{{\operatorname{Ric}}} \def\Hess{{\operatorname{Hess}}}
\def\id{{\operatorname{id}}} \def\trace{\operatorname{tr}}
 
\def\cut{\operatorname{cut}}\def\O{\operatorname{\rm O}}
\def\dist{\operatorname{dist}}\def\I{\operatorname{\mathbb I}}

 \def\di{\displaystyle}
\def\f{\frac} \def\a{\alpha }  
   \def\n{\nabla }
  \def\rr{\rho} 
  \def\beq{\begin{equation}} \def\ff{\frac}\def\nn{\nabla}
    \def\<{\langle} \def\>{\rangle} \def\d{{\rm
        d}}\def\ss{\sqrt} \def\Eig{{\rm
        Eig}}

    \newcommand\newdot{{\kern.8pt\cdot\kern.8pt}}
    \newcommand\mequal{\overset{\rm m}{=}}
    \newcommand\mstrut{{\phantom{.}}}
    \newcommand\bull{{\!\hbox{\bf .}}}

    \title{\bf\Large  Gradient Estimates on Dirichlet and Neumann Eigenfunctions}

    \author[1]{\bf Marc Arnaudon}
    \author[2]{\bf Anton Thalmaier}
    \author[3]{\bf Feng-Yu Wang}
    \affil[1]{Institut de Math\'ematiques de Bordeaux, Universit\'e de Bordeaux,\par
      351 Cours de la Lib\'eration,  F--33405 Talence Cedex, France\par
      \texttt{marc.arnaudon@math.u-bordeaux.fr}\vspace{1em}}

    \affil[2]{Mathematics Research Unit, FSTC, University of Luxembourg, \par
      Maison du Nombre, L--4364 Esch-sur-Alzette, Grand Duchy of Luxembourg\par
      \texttt{anton.thalmaier@uni.lu}\vspace{1em}}

     \affil[3]{Center for Applied Mathematics, Tianjin University, Tianjin 300072, China, and\par
      Department of Mathematics, Swansea University, Singleton Park, SA2 8PP, United Kingdom\par 
\texttt{wangfy@tju.edu.cn}}

\begin{document}
\numberwithin{equation}{section}
\def\theequation{\arabic{section}.\arabic{equation}}

\maketitle

\newcommand\blfootnote[1]{%
  \begingroup
  \renewcommand\thefootnote{}\footnote{#1}%
  \addtocounter{footnote}{-1}%
  \endgroup
}

\blfootnote{This work was supported by Fonds National de la Recherche Luxembourg (Grant No. O14/7628746 GEOMREV) 
and the University of Luxembourg (Grant No. IRP R-AGR-0517-10/AGSDE.  
Feng-Yu Wang acknowledges support in part by NNSFC (11771326, 11431014).}

\begin{abstract} By methods of stochastic analysis on Riemannian
  manifolds, we derive explicit constants $c_1(D)$ and $c_2(D)$ for a
  $d$-dimensional compact Riemannian manifold $D$ with boundary such
  that
  $$c_1(D)\ss\ll \|\phi\|_\infty\le \|\nn \phi\|_\infty \le c_2(D) \ss\ll \|\phi\|_\infty$$ holds
  for any Dirichlet eigenfunction $\phi$ of $-\DD$ with eigenvalue
  $\ll$.  In particular, when $D$ is convex with non-negative Ricci
  curvature, the estimate holds for
  $$c_1(D)= \ff 1{d\e},\quad c_2(D)=\sqrt{\e}\left(\frac{\sqrt{2}}{\sqrt{\pi}}+\frac{\sqrt{\pi}}{4\sqrt{2}}\right).$$
  Corresponding two-sided gradient estimates for Neumann eigenfunctions
  are derived in the second part of the paper.
\end{abstract}


\section{Introduction}

Let $D$ be a $d$-dimensional compact Riemannian manifold with boundary
$\pp D$. We write $(\phi,\ll)\in \Eig(\Delta)$ if $\phi$ is a
Dirichlet eigenfunction of $-\Delta$ in $D$ with eigenvalue
$\ll>0$. According to \cite{G1}, there exist two constants
$c_1(D), c_2(D)>0$ such that \begin{equation}\label{BB}c_1(D)\ss\ll
\|\phi\|_\infty\le \|\nabla \phi\|_\infty \le c_2(D)\ss\ll
\|\phi\|_\infty,\quad (\phi,\ll)\in \Eig(\Delta).\end{equation} An
analogous statement for Neumann eigenfunctions has been derived in
\cite{HSX}.

Concerning Dirichlet eigenfunctions, an explicit upper constant $c_2(D)$ 
can be derived from the uniform gradient  estimate of the Dirichlet semigroup 
in an earlier paper \cite{W04} of the third named author.
More precisely, let $K,\theta\ge 0$ be two constants such that
\begin{equation}\label{CV} \Ric_D\ge -K,\quad  H_{\pp D}\ge -\theta,\end{equation}
where $\Ric_D$ is the
Ricci curvature on $D$ and $H_{\pp D}$ the mean curvature  of $\pp D$. Let
\begin{equation}\label{AA0}\alpha_0= \ff 1 2 \max\big\{\theta, \ss{(d-1)K}\big\}.\end{equation}
Consider the semigroup $P_t=\e^{t\Delta}$ for the Dirichlet Laplacian $\Delta$. 
According to \cite[Theorem 1.1]{W04} where $c=2\alpha_0$, for any nontrivial $f\in \B_b(D)$ and $t>0$,
the following estimate holds:$$\ff{\|\nabla P_t f\|_\infty}{\|f\|_\infty} \le 9.5 \alpha_0 
+ \ff{2\ss{\alpha_0} (1+4^{2/3})^{1/4}\, (1+ 5\times 2^{-1/3})}{(t\pi)^{1/4}}
+ \ff{\ss{1+ 2^{1/3}} \,(1+ 4^{2/3})}{2\ss{t\pi}}=:c(t).$$
Consequently, for any $(\phi,\ll)\in \Eig(\Delta)$,
$$\|\nabla\phi\|_\infty\le \|\phi\|_\infty \inf_{t>0} c(t)\e^{\ll t}.$$
In particular, when $\Ric_D\ge 0$, $H_{\pp D}\ge 0$,
\begin{equation}\label{EX} \|\nabla\phi \|_\infty\le \ff{ \ss{\e\,(1+ 2^{1/3})}\, 
(1+ 4^{2/3})}{\ss{2\pi}}\,\ss\ll\,\|\phi\|_\infty,\quad (\phi,\ll)\in \Eig(\Delta).\end{equation}

In this paper, by using stochastic analysis of the Brownian motion on $D$, we develop two-sided gradient estimates; 
the upper bound given below in \eqref{GF} improves the one in~\eqref{EX}.
Our result will also be valid for $\alpha_0\in \R$ satisfying
\begin{equation}\label{AA0bis}
\f12\Delta \rho_{\partial D} \le \alpha_0\quad\text{outside the focal set,}
\end{equation}
where $\rho_{\partial D}$ is the distance to the boundary. The case $\alpha_0<0$ appears naturally in many situations, 
for instance when $D$ is a closed ball with convex distance to the origin. 
Note that by \cite[Lemma 2.3]{W04}, if under \eqref{CV} we define $\alpha_0$ by~\eqref{AA0} then
condition \eqref{AA0bis} holds as a consequence.

For $x\geq0$, in what follows in the limiting case $x=0$ we use the convention $$
\Big(\ff1{1+x}\Big)^{1 /x}:= \lim_{r\downarrow 0} \Big(\ff1{1+r}\Big)^{1/ r}
=\ff 1 \e.$$

\begin{theorem}\label{T1.1} Let $K, \theta\ge 0$ be two constants such that $\eqref{CV}$ holds and let $\alpha_0$ 
be given by~\eqref{AA0} or more generally satisfy~\eqref{AA0bis}.
Then, for any nontrivial $(\phi,\ll)\in \Eig(\Delta)$,
\begin{equation}\label{Th1.1i}  
\ff{\ll}{\ss{d \e (\ll+K)}}\le \ff \ll {\ss{d(\ll+K)}} \Big(\ff\ll{\ll+K}\Big)^{\ll/ (2K)}\le\ff{\|\nabla\phi\|_\infty}{\|\phi\|_\infty}
\end{equation}
and
\begin{equation}\label{Th1.1s1}
\ff{\|\nabla\phi\|_\infty}{\|\phi\|_\infty}
\le  \begin{cases}\ss{\e(\ll+K)}&\text{if}\quad \ss{\ll+K}\ge  {2}A\\
 \sqrt\e\left(A+ \ff{\ll+K}{4A}\right)&\text{if}\quad \ss{\ll+K}\leq {2}A,
\end{cases}\end{equation}
where 
\begin{equation*}
A:=2\alpha_0^++\ff {\sqrt{2(\ll+K)}}{\ss{\pi}}\,\exp\left(-\frac{\alpha_0^2}{2(\ll+K)}\right).
\end{equation*}
In particular, when $\Ric_D\geq0$, $H_{\pp D}\ge 0$,
\begin{equation}\label{GF}\ff{\ss{\ll}}{\ss{d \e }}\le \ff{\|\nabla\phi\|_\infty}{\|\phi\|_\infty}\\
\le \ss\ll \left(\frac{\sqrt{2\e}}{\sqrt\pi}+\frac{\sqrt{\pi\e}}{4\sqrt{2}}\right),\quad
(\phi,\ll)\in\Eig(\Delta).
\end{equation}
\end{theorem}

\begin{proof} This result follows from Theorem \ref{T2.1} and Theorem
\ref{T2.0a} below in the special case  $V=0$. In this case,
$\Ric_D^V=\Ric_D\ge -K$ is equivalent to \eqref{CD} with
$n=d$. More sophisticated upper bounds are given below in Theorem \ref{T2.0a}. 
\end{proof}

By \eqref{GF}, if $D$ is convex with non-negative Ricci curvature then
$\eqref{BB}$ holds with
$$c_1(D)= \ff{1}{\ss{d \e }},\quad c_2(D)= \frac{\sqrt{2\e}}{\sqrt\pi}+\frac{\sqrt{\pi\e}}{4\sqrt{2}}.$$
To give explicit values of $c_1(D)$ and $c_2(D)$ for positive $K$ or $\theta$, let
$\ll_1>0$ be the first Dirichlet eigenvalue of $-\Delta$ on $D$. Then
Theorem \ref{T1.1} implies that $\eqref{BB}$ holds for
\begin{align*} &c_1(D)= \ff{\ss{\ll_1}}{\ss{d\e (\ll_1+K)}},\\
&c_2(D)= \ff{\ss{\e(\ll_1+K)}}{\ss{\ll_1}}\1_{\{B>2A\}}+
\frac{\sqrt\e}{\sqrt{\ll_1}}\left(2\alpha_0^++\ss{\frac{2(\ll_1+K)}{\pi}}
+ \ff{\ll_1+K}{4\big(2\alpha_0^++{\ss{2(\ll_1+K)/\pi}}\big)}\right) 
\1_{\{B\le 2A\}}\end{align*}
with
$$
B=\ss{\ll_1+K} \quad\text{and}\quad 
A=2\alpha_0^++\ss{\frac{2(\ll_1+K)}{\pi}}.
$$
This is due to the fact that the expression for $c_1(D)$ is an increasing
function of $\ll$ and the expression for $c_2(D)$  a decreasing function of
$\ll$.  Since there exist explicit lower bound estimates on $\ll_1$
(see~\cite{W95} and references within), this gives explicit lower
bounds of $c_1(D)$ and explicit upper bounds of $c_2(D)$.

The lower bound for $\|\nabla\phi\|_\infty$ will be derived by
using It\^o's formula for $|\nabla\phi|^2(X_t)$ where $X_t$ is a
Brownian motion (with drift) on $D$, see Subsection 2.1 for
details.   To derive the  upper bound estimate,
  we
will construct some martingales to reduce $\|\nabla\phi\|_\infty$ to
$\|\nabla\phi\|_{\pp D,\infty}^\mstrut:=\sup_{\pp D} |\nabla \phi|$, and to
estimate the latter in terms of $\|\phi\|_\infty$, see Subsection 2.2
for details.

Next, we consider the Neumann problem. Let $\Eig_N(\Delta)$ be the set of non-trivial eigenpairs $(\phi,\ll)$ for the Neumann eigenproblem, i.e. $\phi$ is non-constant, $\Delta\phi=-\ll \phi$ with $N\phi|_{\pp D}=0$ for the unit inward normal vector field $N$ of $\pp D$. Let $\I_{\pp D}$ be the second fundamental form of $\pp D$,
$$\I_{\pp D}(X,Y)=-\langle\nabla_X N,Y\rangle,\quad X,Y\in T_x\partial M,\ x\in \partial M.$$
  With a concrete choice of the function $f$,
the next theorem implies \eqref{BB}  for $(\phi,\ll)\in \Eig_N(\Delta)$ together with explicit constants $c_1(D), c_2(D)$.

\begin{theorem}\label{T1.2} Let $K,\dd\in \R$ be constants such that
\begin{equation}\label{CV1} \Ric_D\ge -K,\quad \I_{\pp D}\ge -\dd.\end{equation} For $f\in C_b^2(\bar D)$ with $\inf\limits_{D} f=1$ 
and $N\log f|_{\pp D}\ge \dd$, let
\begin{align*} &c_\vv(f)= \sup_D\left\{\ff{4\vv |\nabla\log f|^2}{1-\vv} +K   - 2\Delta\log f\right\},\quad \vv\in (0,1),\\
&K(f)=\sup_{D} \big\{2|\nabla\log f|^2 +K -\Delta\log f\big\}.\end{align*}
Then  for any non-trivial $(\phi,\ll)\in \Eig_N(\Delta)$, we have $\ll+c_\vv(f)>0$ and
\begin{align*}   \sup_{\vv\in (0,1)}   \ff{\vv\ll^2}{d\e(\ll+c_\vv(f))\|f\|_\infty^2}   &\le \sup_{\vv\in (0,1)} \ff{\vv\ll^2}{d(\ll+c_\vv(f))\|f\|_\infty^2} \Big(\ff\ll{\ll+c_\vv(f)}\Big)^{\ll/{c_\vv(f)}}\\
&\le\ff{\|\nabla\phi\|_\infty^2}{\|\phi\|_\infty^2}\le  \ff{2\|f\|_\infty^2(\ll+K(f))}{\pi}\Big(1+\ff{K(f)}\ll\Big)^{\ll/{K(f)}} \\ &\le 2\e\,\|f\|_\infty^2 \ff{\ll+K(f)}{\pi}.\end{align*}\end{theorem}

\begin{proof} Under the conditions \eqref{CV}, Theorem \ref{T3.2} below applies with
$L=\Delta$, $K_V=K$ and $n=d$. The desired estimates are immediate consequences. \end{proof}

When $\pp D$ is convex, i.e.~$\I_{\pp D}\ge 0$, we may take $f\equiv 1$ in Theorem \ref{T1.2} to derive the following result. According to Theorem \ref{T3.1} below, this result also holds for $\pp D=\varnothing$ where $\Eig(\Delta)$ is the set of eigenpairs for the closed eigenproblem.

\begin{corollary} Let $\pp D$ be convex or empty. If $\Ric_D^V\ge -K$ for some constant $K$, then
 for any non-trivial $(\phi,\ll)\in \Eig_N(\Delta)$, we have $\ll+K>0$ and
$$ \ff{\ll^2}{d\e (\ll+K^+)}   \le  \ff{\ll^2}{d(\ll+K)} \Big(\ff\ll{\ll+K}\Big)^{\ll/{K}} \le \ff{\|\nabla\phi\|_\infty^2}{\|\phi\|_\infty^2} \le  \ff{2(\ll+K)}{\pi}\Big(1+\ff{K}\ll\Big)^{\ll/{K}} \le\ff{2\e (\ll+K^+) }{\pi}.$$ \end{corollary}

\section{Proof of Theorem \ref{T1.1}}
In general, we will consider Dirichlet eigenfunctions for the
symmetric operator $L:=\Delta+\nabla V$ on $D$ where $V\in C^2(D)$. We denote
by $\Eig(L)$ the set of pairs $(\phi,\ll)$ where $\phi$ is a
Dirichlet eigenfunction of $-L$ on $D$ with eigenvalue $\ll$.

In the following two subsections, we consider the lower bound and upper bound estimates respectively.

\subsection{Lower bound estimate}
In this subsection we will estimate $\|\nabla\phi\|_\infty$ from below using the following
Bakry-\'Emery cur\-va\-ture-dimension condition: \begin{equation}\label{CD} \ff 1 2 L
|\nabla f|^2 - \<\nabla Lf,\nabla f\>\ge - K |\nabla f|^2 + \ff{(Lf)^2}n,\quad
f\in C^\infty(D),\end{equation} where $K\in\R$, $n\ge d$ are two
constants. When $V=0$, this condition with $n=d$ is equivalent to
$\Ric_D\ge -K$.

\begin{theorem}[Lower bound estimate]\label{T2.1} Assume that $\eqref{CD}$ holds. Then \begin{equation}\label{LB1}
\|\nabla\phi\|_\infty^2 \ge \|\phi\|_\infty^2\sup_{t>0} \ff{\ll^2
  (\e^{Kt}-1)}{n K\e^{(\ll+K)^+t}},\quad (\phi,\ll)\in
\Eig(L).\end{equation} Consequently, for $K^+:= \max\{0,K\}$ there
holds \begin{equation}\label{LB2} \ \|\nabla\phi\|_\infty^2 \ge \ff {\ll^2\|\phi\|_\infty^2} {n(\ll+K^+)} 
\Big(\ff\ll{\ll+K^+}\Big)^{\ll/{K^+}}\ge  \ff{\ll^2\|\phi\|_\infty^2}{n \e
  (\ll+K^+)},\quad (\phi,\ll)\in
\Eig(L).\end{equation}
\end{theorem}

\begin{proof} Let $X_t$ be the diffusion process generated by
$\ff 1 2 L$ in $D$, and let
$$\tau_D:= \inf\{t\ge 0: X_t\in\pp D\}.$$
By It\^o's formula, we have \begin{equation}\label{02.2} \d |\nabla \phi|^2(X_t)= \ff
1 2 L |\nabla \phi|^2(X_t) \,\d t +\d M_t,\quad t\le
\tau_D,\end{equation} for some martingale $M_t$.  By the curvature
dimension condition \eqref{CD} and $L\phi=-\ll\phi$, we obtain
\begin{equation}\label{LMN}\ff 1 2 L  |\nabla \phi|^2 = \ff 1 2 L|\nabla \phi|^2 - \<\nabla L \phi, \nabla\phi\>- \ll |\nabla \phi|^2 \ge -(K+\ll)|\nabla\phi|^2 +\ff {\ll^2}n \phi^2.\end{equation}
Therefore, \eqref{02.2} gives
$$\d |\nabla \phi|^2(X_t)\ge \Big(\ff{\ll^2}n \phi^2- (K+\ll)|\nabla\phi|^2 \Big)(X_t)\,\d t +\d M_t,\quad t\le \tau_D.$$
Hence, for any $t>0$,
\begin{align*} \e^{(K+\ll)^+t}\, \|\nabla\phi\|^2_\infty &\ge \E\Big[ |\nabla\phi|^2(X_{t\land \tau_D}) \e^{(K+\ll)(t\land \tau_D)}\Big]\\
&\ge \ff {\ll^2}{n} \E\left[\int_0^{t\land\tau_D} \e^{(K+\ll)s}
  \phi(X_s)^2\,\d s\right]\\
&= \ff {\ll^2}{n} \E\left[ \int_0^t
  \1_{\{s<\tau_D\}} \e^{(K+\ll)s} \phi(X_s)^2\,\d s\right].\end{align*}
Since $\phi|_{\pp D}=0$ and $L\phi =-\ll\phi$, by Jensen's inequality
we have
$$\E \left[\1_{\{s<\tau_D\}} \phi(X_s)^2\right]\ge \big(\E [\phi(X_{s\land \tau_D})]\big)^2=\e^{-\ll s} \phi(x)^2,$$
where $x=X_0\in D$ is the starting point of $X_t$. Then, by taking $x$
such that $\phi(x)^2= \|\phi\|_\infty^2$, we arrive at
\begin{align*} &\e^{(K+\ll)^+t} \,\|\nabla\phi\|^2_\infty \ge \ff {\ll^2}{n}\int_0^t \e^{(K+\ll)s} \e^{-\ll s} \phi(x)^2\,\d s\\
&\quad= \ff {\ll^2\|\phi\|_\infty^2}{n}\int_0^t \e^{Ks} \,\d s=
\ff{\ll^2(\e^{Kt}-1)}{nK} \|\phi\|_\infty^2.\end{align*} This
completes the proof of \eqref{LB1}.

Since \eqref{CD} holds for $K^+$ replacing $K$, we may and do assume
that $K\ge 0$.  By taking the optimal choice  $t=\ff 1 {K}\log (1+\ff K\ll)$ (by convention $t=\ll^{-1}$ if $K=0$) in \eqref{LB1}, we obtain
 $$\|\nabla\phi\|_\infty^2 \ge \ff{\ll^2\|\phi\|_\infty^2 }{\ll+K} \Big(\ff\ll{\ll+K}\Big)^{\ll/ K} \ge \ff{\ll^2\|\phi\|_\infty^2}{n\e (\ll+K)}.$$
 Hence \eqref{LB2} holds.
\end{proof}

\subsection{Upper bound estimate}

Let $\Ric^V_D= \Ric_D-\Hess_V$. For $K_0,\theta\ge 0$ such that $\Ric_D\ge -K_0$ and 
$H_{\pp D}\ge -\theta$, let
\begin{equation}\label{AA} \alpha= \ff 1 2 \left(\max\big\{ \theta,\ss{(d-1) K_0}\big\} +\|\nabla
  V\|_\infty\right)\end{equation}
We note that $\ff 1 2 L\rr_{\pp D}\le \alpha$ by \cite[Lemma 2.3]{W04}.

\begin{theorem}[Upper bound estimate]\label{T2.0a} Let $K_V,\theta\ge 0$ be constants such that
$$\Ric^V_D\ge -K_V,\quad H_{\pp D}\ge -\theta.$$ 
Let $\alpha\in \R$ be such that
  \begin{equation}\label{AAbisa}  \ff 1 2 L\rho_{\partial D}\le \alpha .\end{equation}

{\rm1.} Assume $\alpha\ge0$.  Then, for any nontrivial
$(\phi,\ll)\in \Eig(L)$,
\begin{equation}\label{T2.0.1a} 
\ff{\|\nabla\phi\|_\infty}{\|\phi\|_\infty}
\le  \begin{cases}\ss{\e(\ll+K_V)}&\text{if}\quad \ss{\ll+K_V}\ge {2}A\\
 \ss{\e}\left(A+ \ff{\ll+K_V}{4A}\right)&\text{if}\quad \ss{\ll+K_V}\leq {2}A,
\end{cases}\end{equation}
where 
\begin{equation}\label{Eq:Aa}
A:=\alpha+\ff {\sqrt{2(\ll+K_V)}}{\ss{\pi}}\,\exp\left(-\frac{\alpha^2}{2(\ll+K_V)}\right) 
+|\alpha|\wedge \ff{\sqrt{2}\alpha^2}{\ss{\pi(\ll+K_V)}}.\end{equation}
In particular, \eqref{T2.0.1a} holds with $A$ replaced by
\begin{equation}\label{Eq:A'a}
A':=2\alpha+\ff {\sqrt{2(\ll+K_V)}}{\ss{\pi}}\,\exp\left(-\frac{\alpha^2}{2(\ll+K_V)}\right) 
.\end{equation}
We also have
\begin{equation}\label{T2.0.1aa} 
\ff{\|\nabla\phi\|_\infty}{\|\phi\|_\infty}
\le \sqrt\e\left( \frac{2\alpha+\ss{2(\ll+K_V)}}{\sqrt\pi}
+\frac{\ll+K_V}4\,\frac{\sqrt\pi}{2\alpha+\ss{2(\ll+K_V)}}\right).
\end{equation}

{\rm2.} Assume $\alpha\le0$. Then, for any nontrivial
$(\phi,\ll)\in \Eig(L)$,
\begin{equation}\label{T2.0.1bb} 
\ff{\|\nabla\phi\|_\infty}{\|\phi\|_\infty}
\le  \begin{cases}\ss{\e(\ll+K_V)}&\text{if}\quad \ss{\ll+K_V}\ge {2}A^*\\
 \ss{\e}\left(A^*+ \ff{\ll+K_V}{4A^*}\right)&\text{if}\quad \ss{\ll+K_V}\leq {2}A^*,
\end{cases}\end{equation}
where 
\begin{equation}\label{Eq:Abb}
A^*:=\ff {\sqrt{2(\ll+K_V)}}{\ss{\pi}}\,\exp\left(-\frac{\alpha^2}{2(\ll+K_V)}\right).\end{equation} 
In particular,
\begin{equation}\label{T2.0.1b} 
\ff{\|\nabla\phi\|_\infty}{\|\phi\|_\infty}
\le  \ss{\ll+K_V}\,\left(\sqrt{\frac2{\pi}}+\frac14\sqrt{\frac{\pi}{2}}\right)\sqrt\e.
\end{equation}
In addition, the following estimate holds:
\begin{equation}\label{T2.0.1_neg} 
\ff{\|\nabla\phi\|_\infty}{\|\phi\|_\infty}
\le  \begin{cases}\ss{\e(\ll+K_V)}&\text{if}\quad \ss{\ll+K_V}\ge  2\sqrt{\e}\hat A\\
 \e \hat A+ \ff{\ll+K_V}{4\hat A}&\text{if}\quad \ss{\ll+K_V}<  2\sqrt{\e}\hat A,
\end{cases}\end{equation}
where 
\begin{equation}\label{Eq:A_neg}
\hat A:=\alpha+\ff{\ss{2\ll}}{\ss\pi} \e^{-\ff{\alpha^2}{2\ll}}+ |\alpha|\land \ff{\sqrt{2}\alpha^2}{\ss{\pi\ll}}.
\end{equation} \end{theorem}

The strategy to prove Theorem \ref{T2.0a} will be to first estimate $\|\nabla\phi\|_\infty$ in terms
of $\|\phi\|_\infty$ and $\|\nabla \phi\|_{\pp D,\infty}^\mstrut$ (see estimate \eqref{Eq:Est} below) where
$\|f\|_{\pp D,\infty}^\mstrut := \|\1_{\pp D}f\|_\infty$ for a function
$f$ on $D$.
The this end we construct appropriate martingales in terms of 
$\phi$ and $\nabla\phi$.

We start by recalling the necessary facts about the diffusion
process generated by $\frac12L$, see for instance
\cite{Arnaudon-Thalmaier-Wang:09,Hsu}. 
For any $x\in D$, the diffusion
$X_t$ solves the SDE \begin{equation}\label{E1} \d X_t= \ff 1 2 \nabla V(X_t)\,\d t +
u_t\circ \d B_t,\quad X_0=x,\ t\le \tau_D,
\end{equation}
where $B_t$ is a $d$-dimensional Brownian motion, $u_t$ is the
horizontal lift of $X_t$ onto the ortho\-normal frame bundle $\O(D)$
with initial value $u_0\in \O_x(D),$ and
$$\tau_D:=\inf \{t\ge 0: X_t\in\pp D\}$$ is the hitting time of $X_t$ to the boundary $\pp D$.
Setting $Z:=\nabla V$, we have \begin{equation}\label{E1_lift} \d u_t= \ff 1 2
Z^*(u_t)\,\d t + \sum_{i=1}^d H_i(u_t)\circ \d B_t^i
\end{equation}
where $Z^*(u):=h_u(Z_{\pi(u)})$ and $H_i(u):=h_u(ue_i)$ are defined by
means of the horizontal lift $h_u\colon T_{\pi(u)}D\to T_u\O(D)$ at
$u\in \O(D)$. Note that formally
$h_{u_t}(u_t\circ \d B_t)=\sum_ih_{u_t}(u_te_i)\circ \d
B_t^i=\sum_iH_i(u_t)\circ \d B_t^i$.

For $f\in C^\infty(D)$, let $a:=\d f\in\Gamma(T^*D)$.  Setting
$m_t:=u_t^{-1}a(X_t)$, we see by It\^o's formula that
\begin{equation}\label{Eq:1}
  \d m_t\mequal\frac12 u_t^{-1}(\square a+\nabla_Za)(X_t)\,\d t
\end{equation}
where $\square a=\trace\nabla^2a$ denotes the so-called connection (or
rough) Laplacian on 1-forms and $\overset{\rm m}{=}$ equality modulo
the differential of a local martingale.

Denote by $Q_t\colon T_xD\to T_{X_t}D$ the solution, along the paths
of $X_t$, to the covariant ordinary differential equation
$$DQ_t=-\frac12 (\Ric_D^V)^\sharp Q_t\,\d t,\quad Q_0=\id_{T_xD},\ t\le \tau_D,$$
where $D:=u_t{\d} u_t^{-1}$ and where by definition
$$(\Ric^V_D)^\sharp v= \Ric^V_D(\newdot,v)^\sharp ,\quad v\in T_xD.$$
Thus, condition $\Ric_D^V\ge -K_V$ implies \begin{equation}\label{VT} |Q_tv|\le
\e^{\ff {K_V} 2 t}\,|v|,\quad t\le \tau_D.
\end{equation}
Finally, note that for any smooth function $f$ on $D$, we have by the
Weitzenb\"ock formula:
\begin{align}
  \d\big(\Delta+Z\big)f&=\d\big(-\d^*\d f+(\d f)Z\big)\notag\\
                       &=\Delta^{(1)}\d f+\nabla_Z \d f+\langle\nabla_\bull Z,\nabla f\rangle\notag\\
                       &=(\square+\nabla_Z )(\d f)-\Ric_D^V(\newdot,\nabla f)\notag\\
                       &=\big(\square-\Ric_D^V+\nabla_Z\big)(\d f)\label{Eq:Comm}
\end{align}
where $\Delta^{(1)}$ denotes the Hodge-deRham Laplacian on 1-forms.

Now let $(\phi,\ll)\in \Eig(L)$, i.e. $L\phi=-\ll\phi$, where
$L=\Delta+Z$.  For $v\in T_xD$, consider the
process $$n_t(v):=(\d\phi)(Q_tv).$$ Then
$$n_t(v)=\langle\nabla\phi(X_t),Q_tv\rangle=\langle u_t^{-1}(\nabla\phi)(X_t),u_t^{-1}Q_tv\rangle.$$
Using \eqref{Eq:1}, we see by It\^o's formula and formula
\eqref{Eq:Comm} that
$$\d n_t(v)\mequal \frac12(\square \d\phi+\nabla_Z\d\phi)(X_t)\,Q_tv\,\d t+\d\phi(X_t)(DQ_tv)\,\d t=
-\frac\lambda2 n_t(v)\,\d t.$$ It follows that
\begin{equation}\label{Eq:Mart}
\e^{\ll t/2}\,n_t(v)  =\e^{\ll t/2} \,\<\nabla\phi(X_t), Q_tv\>,\quad t\le \tau_D,
\end{equation}
is a martingale.

\begin{lemma} Let $(\phi,\ll)\in \Eig(L)$. We keep the notation from above. 
Then, for any function
$h\in C^1([0,\infty);\R)$, the process
\begin{align}\label{Cruc_Mart}N_t(v):= h_t \,\e^{\ll t/2} \,\<\nabla\phi(X_t), Q_tv\>- \e^{\ll t/2} \,\phi(X_t) \int_0^t \<\dot h_s Q_sv, u_s\d B_s\>,\quad t\le\tau_D,
\end{align} 
is a martingale. In particular, for fixed $t>0$ and $h\in C^1([0,t];[0,1])$ 
monotone such that
$h_0=1$ and $h_t=0$, we have
\begin{align}\label{Eq:Est}
\|\nabla\phi\|_\infty &\le\  
 \|\nabla \phi\|_{\pp D,\infty}^\mstrut\,\P\{t>\tau_D\} \,\e^{(\ll+K_V)^+t/2}\notag\\ 
&\quad+\|\phi\|_\infty \,\e^{\ll t/2} \,\P\{t\leq\tau_D\}^{1/2}\,\bigg(\int_0^t |\dot h_s|^2 \e^{K_V s}
\,\d s\bigg)^{1/2}.\end{align}
\end{lemma}

\begin{proof}Indeed, from  \eqref{Eq:Mart} we deduce that
$$h_t\,\e^{\ll t/2} \,\<\nabla\phi(X_t), Q_tv\> -\int_0^t \dot h_s \,\e^{\ll s/2}\,\<\nabla\phi(X_s), Q_sv\> \,\d s,\quad t\le \tau_D,$$ is a martingale as well.
By the formula
$$\e^{\ll t/2}\, \phi(X_t) = \phi(X_0)+ \int_0^t \e^{\ll s/2}\,\langle \nabla\phi(X_s),u_s\d B_s\rangle$$ 
we see then that $N_t(v)$ is a martingale. To check inequality \eqref{Eq:Est}, 
we deduce from
the martingale property of
$\{N_{s\land\tau_D}(v)\}_{s\in [0, t]}$ that
\begin{align*}
\|\nabla\phi\|_\infty &\le\  
 \|\nabla \phi\|_{\pp D,\infty}^\mstrut\,\E\left[\1_{\{t>\tau_D\}} \,\e^{\ll \tau_D/2} \,|h_{\tau_D}| \,
|Q_{\tau_D}|  \right]\\
&\quad +\|\phi\|_{\infty}^\mstrut\,\e^{\ll t/2}\,\E\left[\1_{\{t\le \tau_D\}} 
 \sup_{|v|\leq1}\left(\int_0^t \<\dot h_s\, Q_sv, u_s\d B_s\>\right)^2\right]^{1/2}.
\end{align*}
The claim follows by using \eqref{VT}.
\end{proof}

To estimate the boundary norm $\|\nabla\phi\|_{\pp D,\infty}^\mstrut$, we shall
compare $\phi(x)$ and
$$\psi(t,x):= \P(\tau_D^x>t),\quad t>0,$$ for small $\rr^\mstrut_{\pp D}(x):= {\rm dist} (x,\pp D)$. 
Let $P_t^D$  be the Dirichlet semigroup generated by $\ff 1 2 L$.
Then $$\psi(t,x) = P_t^D \1_D (x),$$ so that \begin{equation}\label{PT}\pp_t
\psi(t,x)= \ff 1 2 L \psi(t,\newdot)(x),\quad t>0.
\end{equation}

\begin{lemma}\label{L2.2} For any $(\phi,\ll)\in\Eig(L)$,
\begin{equation}\label{ETS}\|\nabla \phi\|_{\pp D,\infty}^\mstrut \le
\|\phi\|_\infty\, \inf_{t>0} \e^{\ll t/2} \,\|\nabla
\psi(t,\newdot)\|_{\pp D,\infty}^\mstrut.
\end{equation}
\end{lemma}

\begin{proof} To prove \eqref{ETS}, we fix $x\in\pp D$. For
small $\vv>0$, let $x^\vv= \exp_x(\vv N)$, where $N$ is the inward
unit normal vector field of $\pp D$. Since $\phi|_{\pp D}=0$ and
$\psi(t,\newdot)|_{\pp D}=0$, we have \begin{equation}\label{MG}|\nabla \phi(x)|= |N
\phi(x)|=\lim_{\vv\to 0} \ff{|\phi(x^\vv)|}\vv,\quad
|\nabla\psi(t,\newdot)(x)|= \lim_{\vv\to 0}
\ff{|\psi(t,x^\vv)|}\vv.\end{equation}
Let $X_t^\vv$ be the $L$-diffusion starting at $x^\vv$ and
$\tau_D^\vv$ its first hitting time of $\partial D$. Note that
$$N_t:=\phi(X_{t\wedge\tau_D^\vv}^\vv)\,\e^{\lambda (t\wedge\tau_D^\vv)/2},\quad t\geq0,$$
is a martingale.
Thus, for each fixed $t>0$, we can estimate as follows:
\begin{align*}
|\nabla \phi(x)|
&= \lim_{\vv\to 0} \ff{|\phi(x^\vv)|}\vv \\
 &= \lim_{\vv\to 0} \frac {\left|\E[\phi(X_t^\vv)\, \1_{\{t<\tau_D^\vv\}} ]\,
 \e^{\lambda (t\wedge\tau_D^\vv)/2}\right|}{\vv}\\
&\leq \|\phi\|_\infty\, \e^{\ll t/2} \lim_{\vv\to 0}
\frac {\E[\1_{\{t<\tau_D^\vv\}} ]}{\vv}\\
&\leq \|\phi\|_\infty\, \e^{\ll t/2}
\lim_{\vv\to 0}
\ff{\psi(t,x^\vv)}\vv\\
&=\|\phi\|_\infty\, \e^{\ll t/2} \,|\nabla \psi(t,\newdot)|(x).
\end{align*}
Taking the infimum over $t$ gives the claim.
\end{proof}

We now work out an explicit estimate for $\|\nabla \psi(t,\newdot)\|_{\partial D,\infty}$. 
Let $\cut(D)$ be the
cut-locus of $\pp D$, which is a zero-volume closed subset of $D$ such
that $\rho^\mstrut_{\pp D}:=\dist(\newdot, \pp D)$ is smooth in
$D\setminus \cut(D)$.

\begin{proposition}\label{P6.3}
Let $\alpha\in\R$ such that 
\begin{equation}\label{Eq:alpha}\f12 L\rho_{\partial_D}\le \alpha.
\end{equation} Then
\begin{align}
\|\n \psi(t,\newdot)\|_{\pp D,\infty}^\mstrut
&\le\a+\f{\sqrt2}{\sqrt{\pi t}}+\int_0^t\f{1-\e^{-\f{\a^2s}{2}}}{\sqrt{2\pi s^3}}\,\d s\notag\\
&\le \alpha+\ff{\ss 2}{\ss{\pi t}}e^{-\f{\alpha^2t}{2}} +\min\bigg\{ |\a|,\ \ff{\alpha^2\ss{2t}}{\ss{\pi}}\bigg\}, \label{E6.5.0}
\end{align}
and
\begin{equation}\label{E6.5.0bis}
\|\n \psi(t,\newdot)\|_{\pp D,\infty}^\mstrut\le \frac{\sqrt{2}}{\sqrt{ \pi t}}+\alpha+\frac{\sqrt{t}}{\sqrt{2 \pi}}\alpha^2
\end{equation}
\end{proposition}

Notice that by \cite[Lemma 2.3]{W04} the condition $\ff 1 2 L\rr_{\pp D}\le \alpha$ holds for $\alpha$ defined by \eqref{AA}.

\begin{proof} Let $x\in D$ and let $X_t$ solve SDE \eqref{E1}.
As shown in \cite{Kendall}, $(\rr^\mstrut_{\pp D}(X_t))_{t\le\tau_D}$ is a semimartingale satisfying
\begin{equation}
 \label{E6.5.0.0}
 \rho^\mstrut_{\partial D}(X_t)=\rho^\mstrut_{\partial D}(x)+b_t+\f12\int_0^t L\rho^\mstrut_{\partial D}(X_s)\,\d s -l_t,\quad t\le \tau_D,
\end{equation}
where $b_t$ is a real-valued Brownian motion starting at $0$, and $l_t$ a non-decreasing process which increases only when $X_t^x\in \cut(D)$.
Setting $\vv=\rho^\mstrut_{\partial D}(x)$, we deduce from  \eqref{E6.5.0.0} together with $\ff 1 2 L\rr_{\pp D}\le \alpha,$ that
\begin{equation}
 \label{E6.5.0.1}
 \rho^\mstrut_{\partial D}(X_t(x))\le Y^{\a}_t(\vv):=\vv+b_t+\a t,\quad t\le\tau_D.
\end{equation}
  Consequently, letting $T^\a(\vv)$ be the first hitting time of $0$ by $Y^\a_t(\vv)$, we obtain
\begin{equation}
 \label{E6.5.0.2}
 \psi(t,x)\le \P(t<T^\a(\vv)).
\end{equation}
  On the other hand, since $\psi(t,\newdot)$ vanishes
on the boundary and is positive in $D$, we have for all $y\in \partial D$
\begin{equation}
 \label{E6.5.0.4}
 |\n \psi(t,y)|=\lim_{x\in D,\, x\to y}\f{\psi(t,x)}{\rho^\mstrut_{\partial D}(x)}.
\end{equation}
Hence, by \eqref{E6.5.0.2}, to prove the first inequality in \eqref{E6.5.0} it is enough to establish that
\begin{equation}
 \label{E6.5.0.3}
 \limsup_{\vv\downarrow 0} \ff{\P(t<T^\a(\vv))}\vv\le\a+\f{\sqrt2}{\sqrt{\pi t}}+\int_0^t\f{1-\e^{-\f{\a^2s}{2}}}{\sqrt{2\pi s^3}}\,\d s.
\end{equation}
It is well known that the (sub-probability) density $f_{\a,\vv}$ of  $T^{\a}(\vv)$ is
\begin{equation}
\label{E6.5.1}
f_{\a,\vv}(s)=\f{\vv\exp\big({-(\vv+\a s)^2}/{(2s)}\big)}{\sqrt{2\pi s^3}},
\end{equation}
which can be obtained by the reflection principle for $\a=0$ and the Girsanov transform for $\a\ne 0$.
Thus
\begin{equation}
\label{E6.5.2}\begin{split}
\P(t\ge T^{\a}(\vv))
&=\vv\int_0^t \f{\exp\big({-(\vv+\a s)^2}/{(2s)}\big)}{\sqrt{2\pi s^3}}\,\d s \\
&=\vv\exp(-\a \vv)\int_0^t\f{\e^{-{\a^2 s}/{2}}}{\sqrt{2\pi s^3}}\exp\left(-\f{\vv^2}{2s}\right)\,\d s\\
  &=\exp(-\a \vv)\int_0^{2t/\vv^2}\f{\e^{-1/r}}{\sqrt{\pi r^3}}\exp\left(-\f{\a^2\vv^2r}{4}\right)\,\d r,
 \end{split}
\end{equation}
where we have made the change of variable $r=2s/\vv^2$. With the change of variable $v=1/r$ we easily check that
\begin{equation}
 \label{E6.5.2.2}
 \int_0^\infty r^{-3/2}\e^{-1/r}\,\d r=\Gamma(1/2)=\sqrt{\pi},
\end{equation}
and this allows to write
\begin{equation}
 \label{E6.5.2.3}
 \P(t\ge T^{\a}(\vv))=\exp(-\a \vv)\left(1-\int_{2t/\vv^2}^\infty\f{\e^{-1/r}}{\sqrt{\pi r^3}}\,\d r-
 \int_0^{2t/\vv^2}\f{\e^{-1/r}}{\sqrt{\pi r^3}}\left(1-\e^{{-\a^2 \vv^2r}/{4}}\right)\,\d r\right).
\end{equation}
As $\vv\to 0$,
\begin{align*}
 \int_{2t/\vv^2}^\infty\f{\e^{-1/r}}{\sqrt{ r^3}}\,\d r= \int_{2t/\vv^2}^\infty\f{1}{\sqrt{ r^3}}\,\d r +
{\rm o}(\vv)=\f{\vv\sqrt2}{\sqrt{t}}+{\rm o}(\vv),
\end{align*}
and with change of variable $s=\ff 1 2\vv^2 r$
\begin{align*}
 \int_0^{2t/\vv^2}\f{\e^{-1/r}}{\sqrt{\pi r^3}}\left(1-\e^{-\f{\a^2 \vv^2r}{4}}\right)\,\d r&=
 \vv\int_0^t\f{\e^{-\f{\vv^2}{2s}}}{\sqrt{2\pi s^3}}\left(1-\e^{-\f{\a^2 s}{2}}\right)\,\d s\\
 &= \vv\int_0^t\f{1-\e^{-\f{\a^2 s}{2}}}{\sqrt{2\pi s^3}}\,\d s +{\rm o}(\vv)
\end{align*}
by monotone convergence.
Combining these with $\e^{-\alpha\vv}=1-\alpha\vv +{\rm o}(\vv)$, we deduce from
\eqref{E6.5.2.3} that
\begin{equation}
 \label{E6.5.3}
 \P(t\ge T^{\a}(\vv))=1-\vv\left(\a+\f{\sqrt2}{\sqrt{\pi t}}+\int_0^t\f{1-\e^{-\f{\a^2 s}{2}}}{\sqrt{2\pi s^3}}\,ds\right)+{\rm o}(\vv)
\end{equation}
which yields~\eqref{E6.5.0.3}.

Next, an integration by parts yields
\begin{equation}\label{B*1} \int_0^t\f{1-\e^{-\f{\a^2 s}{2}}}{\sqrt{2\pi s^3}}\,ds=\frac{\alpha^2}{\sqrt{ 2\pi }}\int_0^t\frac1{\sqrt{u}}\,\e^{-\frac{\alpha^2 u}{2}}\, \d u-\frac{\sqrt{2}}{\sqrt{ \pi t}}\left(1-\e^{-\frac{\alpha^2 t}{2}}\right).
\end{equation}
With the change of variable $\di s=|\alpha|\ss{\ff{u}{t}}$ in the first term in the right we obtain
\begin{equation}\label{B*2} \frac{\alpha^2}{\sqrt{ 2\pi }}\int_0^t\frac1{\sqrt{u}}\,\e^{-\frac{\alpha^2 u}{2}}\, \d u
=|\alpha|\ss{\ff{2t}{\pi}}\int_0^{|\alpha|}\e^{-\ff{s^2t}{2}}\,\d s.
\end{equation}
We arrive at
\begin{equation}
  f(\alpha):=\a+\f{\sqrt2}{\sqrt{\pi t}}+\int_0^t\f{1-\e^{-\f{\a^2 s}{2}}}{\sqrt{2\pi s^3}}\,\d s=
  \frac{\sqrt{2}}{\sqrt{ \pi t}}\e^{-\frac{\alpha^2 t}{2}}+\alpha+|\alpha|\ss{\ff{2t}{\pi}}\int_0^{|\alpha|}\e^{-\ff{s^2t}{2}}\,\d s.
\end{equation}
Bounding $\di \ss{\ff{2t}{\pi}}\int_0^{|\alpha|}\e^{-\ff{s^2t}{2}}\,\d s$ by $\di \ss{\ff{2t}{\pi}}\int_0^{\infty}\e^{-\ff{s^2t}{2}}\,\d s=1$, respectively bounding $\e^{-\ff{s^2t}{2}}$ by $1$ in the integral, yields~\eqref{E6.5.0}.

  The function
$$
f(\alpha)=
  \frac{\sqrt{2}}{\sqrt{ \pi t}}\e^{-\frac{\alpha^2 t}{2}}+\alpha+|\alpha|\ss{\ff{2t}{\pi}}\int_0^{|\alpha|}\e^{-\ff{s^2t}{2}}\,\d s
$$ is smooth and an easy computation shows that
\begin{equation}\label{fTE}
f(0)=\frac{\sqrt{2}}{\sqrt{ \pi t}},\quad f'(0)=1,\quad f''(\alpha)=\frac{\sqrt{2t}}{\sqrt{ \pi}}\e^{-\ff{\alpha^2t}{2}}
\end{equation}
Using the fact that $f(\alpha)-\alpha$ is even, we also get
\begin{equation}\label{fTE2}
f(\alpha)=\frac{\sqrt{2}}{\sqrt{ \pi t}}+\alpha+\int_0^{|\alpha|}\frac{\sqrt{2t}}{\sqrt{ \pi}}\e^{-\ff{s^2t}{2}}s\,ds\le
\frac{\sqrt{2}}{\sqrt{ \pi t}}+\alpha+\frac{\sqrt{t}}{\sqrt{2 \pi}}\alpha^2.
\end{equation}
  which yields~\eqref{E6.5.0bis}.
  \end{proof}

\begin{remark}
One could use estimate \eqref{Eq:Est} (optimizing the right-hand side with respect to 
$t$) together with Lemma \ref{L2.2} (again optimizing with respect to 
$t$)  to estimate  $\|\nabla\phi\|_\infty^\mstrut$ in terms of $\|\phi\|_\infty^\mstrut$. 
We prefer to combine the two steps. 
\end{remark}

\begin{lemma}\label{L2.1a} Assume $\Ric_D^V\ge -K_V$ for some constant
$K_V\in\R$. Let $\alpha$ be determined by \eqref{Eq:alpha}.
\begin{enumerate}[\rm(a)]
\item If $\alpha\ge0$, then for any $(\phi,\ll)\in \Eig(L)$,
  \begin{equation*}
    \|\nabla\phi\|_\infty \le \inf_{t>0}\max_{\vv\in [0,1]}\e^{\frac{(\ll+K_V^+)t}2}\, \left\{\vv\, \left(
        \alpha+\ff{\ss 2}{\ss{\pi t}}e^{-\f{\alpha^2t}{2}} +\min\bigg( |\a|,\ \ff{\alpha^2\ss{2t}}{\ss{\pi}}\bigg)
      \right)+\ss{\frac{1-\vv}{t}}\right\}\|\phi\|_{\infty}^\mstrut,\end{equation*}
  as well as
  \begin{equation*}
    \|\nabla\phi\|_\infty \le \inf_{t>0}\max_{\vv\in [0,1]}
    \e^{(\ll+K_V^+)t/2} \,\left\{\vv\, \left(
        \alpha+\sqrt{\frac2{\pi t}} +\frac{\sqrt{t}}{\sqrt{2 \pi}}\alpha^2
      \right)+\ss{\frac{1-\vv}{t}}\right\}\|\phi\|_{\infty}^\mstrut\end{equation*}
  and
  \begin{equation*}
    \|\nabla\phi\|_\infty \le \inf_{t>0}\max_{\vv\in [0,1]}
    \e^{(\ll+K_V^+)t/2} \,\left\{\vv\, \left(
        2\alpha+\sqrt{\frac2{\pi t}} 
      \right)+\ss{\frac{1-\vv}{t}}\right\}\|\phi\|_{\infty}^\mstrut.
  \end{equation*}

\item If $\alpha\le0$, then \begin{equation*} \|\nabla\phi\|_\infty
    \le \inf_{t>0}\max_{\vv\in [0,1]} \e^{(\ll+K_V^+)t/2}
    \,\left\{\vv\, \sqrt{\frac2{\pi
          t}}\,e^{-\f{\alpha^2t}{2}}+\ss{\frac{1-\vv}{t}}\right\}\|\phi\|_{\infty}^\mstrut.
  \end{equation*}
  In particular,
  \begin{equation*}
    \|\nabla\phi\|_\infty \le \inf_{t>0}\max_{\vv\in [0,1]}
    \e^{(\ll+K_V^+)t/2} \,\left\{\vv\, 
      \sqrt{\frac2{\pi t}}+\ss{\frac{1-\vv}{t}}\right\}\|\phi\|_{\infty}^\mstrut.
  \end{equation*}
\end{enumerate}
\end{lemma}

\begin{proof} 
For fixed $t>0$ in \eqref{Cruc_Mart}, we take $h\in C^1([0,t];[0,1])$ such that
$h_0=1$ and $h_t=0$. Then, by the martingale property of
$\{N_{s\land\tau_D}(v)\}_{s\in [0, t]}$, we obtain
\begin{align}
  &|\nabla_v\phi|(x) = |N_0(v)|= |\E N_{t\land\tau_D}(v)| \notag \\
  &= \left| \E\left[\1_{\{t>\tau_D\}} \,\e^{\ll \tau_D/2} \,h_{\tau_D} \<\nabla
    \phi(X_{\tau_D}),Q_{\tau_D}v\> - \1_{\{t\le \tau_D\}} \e^{\ll t/2}
  \phi(X_t) \int_0^t \<\dot h_s\, Q_sv, u_s\d B_s\>\right]\right|.\label{Est}
\end{align}
Note that using \eqref{VT} along with Lemma \ref{L2.2} we may estimate
\begin{align*}
&\left| \E\left[\1_{\{t>\tau_D\}} \,\e^{\ll \tau_D/2} \,h_{\tau_D} \<\nabla
    \phi(X_{\tau_D}),Q_{\tau_D}v\> \right]\right|\\
&\qquad\leq
\E\left[\1_{\{t>\tau_D\}} \,\e^{\ll \tau_D/2} \,|h_{\tau_D}| \, \|\nabla \phi\|_{\pp D,\infty}^\mstrut\,
\e^{K_V\tau_D/2}|v|  \right]\\
&\qquad\leq
\E\left[\1_{\{t>\tau_D\}} \,\e^{\ll \tau_D/2} \,|h_{\tau_D}|\,   \|\phi\|_{\infty}^\mstrut  \,\|\nabla \psi(t-\tau_D,\newdot)\|_{\pp D,\infty}^\mstrut\,
\e^{\ll (t-\tau_D)/2}\,\e^{K_V\tau_D/2}\,|v|  \right]\\
&\qquad=
\E\left[\1_{\{t>\tau_D\}} \,|h_{\tau_D}|\,   \|\phi\|_{\infty}^\mstrut  \,\|\nabla \psi(t-\tau_D,\newdot)\|_{\pp D,\infty}^\mstrut\,
\e^{\ll t/2}\,\e^{K_V\tau_D/2}\,|v|  \right]\\
&\qquad\leq \e^{(\ll+K_V^+)t/2}\,   \|\phi\|_{\infty}^\mstrut\,  \E\left[\1_{\{t>\tau_D\}} \,|h_{\tau_D}|\,\|\nabla \psi(t-\tau_D,\newdot)\|_{\pp D,\infty}^\mstrut\,|v|  \right],
\end{align*}
as well as
\begin{align*}\E\left[\1_{\{t\le \tau_D\}} \,\e^{\ll t/2}
  \phi(X_t) \int_0^t \<\dot h_s\, Q_sv, u_s\d B_s\>\right]
\leq\e^{\ll t/2} \,\|\phi\|_\infty \,\P\{t\leq\tau_D\}^{1/2}\,\bigg(\int_0^t |\dot h_s|^2 \e^{K_V s}
\,\d s\bigg)^{1/2}.
\end{align*}
Taking
$$h_s= \ff{t-s}{t},\quad s\in [0,t],$$ 
we obtain thus from \eqref{Est} 
\begin{align*}|\nabla \phi(x)| &\le \frac{\e^{(\ll+K_V^+)t/2}}t\,\|\phi\|_{\infty}^\mstrut \,\E\left[\1_{\{t>\tau_D\}} \,(t-\tau_D)\,\|\nabla \psi(t-\tau_D,\newdot)\|_{\pp D,\infty}^\mstrut  \right]\\
&\quad+\e^{\ll t/2} \,\|\phi\|_\infty \,\P\{t\leq\tau_D\}^{1/2}\,\frac1t\bigg(\frac{\e^{K_V^+ t}-1}
{K_V^+}\bigg)^{1/2}.\end{align*} 
Note that
$$\frac{\e^{K_V^+ t}-1}{K_V^+}\leq t\e^{K_V^+ t}.$$
(i) By \eqref{E6.5.0}, assuming that $\alpha\ge0$, we have on $\{t>\tau_D\}$:
\begin{align*}
\frac{t-\tau_D}t\,\|\n \psi(t-\tau_D,\newdot)\|_{\pp D,\infty}^\mstrut&\le 
\a\frac{t-\tau_D}t+\f{\sqrt2}{\sqrt{\pi}}\frac{\sqrt{t-\tau_D}}t
+\frac{t-\tau_D}t\int_0^{t-\tau_D}\f{1-\e^{-\f{\a^2s}{2}}}{\sqrt{2\pi s^3}}\,\d s\\
&\le 
\a+\f{\sqrt2}{\sqrt{\pi t}}+\int_0^{t}\f{1-\e^{-\f{\a^2s}{2}}}{\sqrt{2\pi s^3}}\,\d s\\
&\le \alpha+\ff{\ss 2}{\ss{\pi t}}e^{-\f{\alpha^2t}{2}} +\min\bigg\{ \a,\ \ff{\alpha^2\ss{2t}}{\ss{\pi}}\bigg\}. 
\end{align*}
Thus, letting $\vv=\P(t>\tau_D)$, we obtain
\begin{align*}|\nabla \phi(x)| \le \e^{(\ll+K_V^+)t/2}\,\|\phi\|_{\infty}^\mstrut \,\left[\vv\, \left(
\alpha+\ff{\ss 2}{\ss{\pi t}}e^{-\f{\alpha^2t}{2}} +\min\bigg\{ \a,\ \ff{\alpha^2\ss{2t}}{\ss{\pi}}\bigg\}
\right)+\ss{\frac{1-\vv}{t}}\right].
\end{align*}
(ii) Still under the assumption $\alpha\ge0$, this time using estimate \eqref{E6.5.0bis}, 
we have on $\{t>\tau_D\}$:
\begin{equation*}
\|\n \psi(t-\tau_D,\newdot)\|_{\pp D,\infty}^\mstrut\le \frac{\sqrt{2}}{\sqrt{ \pi (t-\tau_D)}}
+\alpha+\frac{\sqrt{t-\tau_D}}{\sqrt{2 \pi}}\alpha^2,
\end{equation*}
and thus letting $\vv=\P(t>\tau_D)$, we get
\begin{align*}|\nabla \phi(x)| &\le \frac{\e^{(\ll+K_V^+)t/2}}t\,\|\phi\|_{\infty}^\mstrut \,\E\left[\1_{\{t>\tau_D\}} \left(
\sqrt{\frac2{\pi}} \sqrt{t-\tau_D}+\alpha(t-\tau_D)+\frac{(t-\tau_D)^{3/2}}{\sqrt{2 \pi}}\alpha^2
\right)\right]\\
&\quad+\e^{\ll t/2} \,\|\phi\|_\infty \,\P\{t\leq\tau_D\}^{1/2}\,\frac1t\bigg(\frac{\e^{K_V^+ t}-1}
{K_V^+}\bigg)^{1/2}\\
&\le \e^{(\ll+K_V^+)t/2}\,\|\phi\|_{\infty}^\mstrut \,\left[\vv\, \left(
\sqrt{\frac2{\pi t}} +\alpha+\frac{\sqrt{t}}{\sqrt{2 \pi}}\alpha^2
\right)+\ss{\frac{1-\vv}{t}}\right].
\end{align*}
(iii) In the case $\alpha\le0$, we get from \eqref{E6.5.0} in a similar way:
\begin{align*}|\nabla \phi(x)| \le \e^{(\ll+K_V^+)t/2}\,\|\phi\|_{\infty}^\mstrut \,\left\{\vv\, 
\ff{\ss 2}{\ss{\pi t}}\,e^{-\f{\alpha^2t}{2}} 
+\ss{\frac{1-\vv}{t}}\right\}.
\end{align*}
This concludes the proof of Lemma \ref{L2.1a}.
\end{proof}

\begin{proposition}\label{L2.1b} We keep the assumptions of Lemma \ref{L2.1a}.\smallskip

\noindent{\rm(a)} If $\alpha\ge0$, then for any $(\phi,\ll)\in \Eig(L)$, 
\begin{align*}
\|\nabla\phi\|_\infty \le \sqrt{\e}\max_{\vv\in [0,1]}&\left\{\vv\, \left(
\alpha+\ff {\sqrt{2(\ll+K_V^+)}}{\ss{\pi}}\,\exp\left({-\frac{\alpha^2}{2(\ll+K_V^+)}}\right) 
+\min\bigg( |\a|,\, \ff{\sqrt{2}\alpha^2}{\ss{\pi(\ll+K_V^+)}}\bigg)
\right)\right.\\
&+\left.\ss{1-\vv}\,\sqrt{(\ll+K_V^+)}\right\}\|\phi\|_{\infty}^\mstrut,
\end{align*}
as well as
\begin{equation*}
\|\nabla\phi\|_\infty \le \sqrt{\e}\max_{\vv\in [0,1]}
\left\{\vv\, \left(
\alpha+\ff{\sqrt{2(\ll+K_V^+)}}{\ss{\pi}} +\ff{\alpha^2}{\ss{2\pi(\ll+K_V^+)}}
\right)+\ss{1-\vv}\,\sqrt{(\ll+K_V^+)}\right\}\|\phi\|_{\infty}^\mstrut
\end{equation*}
and
\begin{equation*}
\|\nabla\phi\|_\infty \le \sqrt{\e}\max_{\vv\in [0,1]}
\left\{\vv\, \left(
2\alpha+\ff{\sqrt{2(\ll+K_V^+)}}{\ss{\pi}}
\right)+\ss{1-\vv}\,\sqrt{(\ll+K_V^+)}\right\}\|\phi\|_{\infty}^\mstrut
\end{equation*}

\noindent{\rm(b)} If $\alpha\le0$, then \begin{equation*}
\|\nabla\phi\|_\infty \le \sqrt{\e}\max_{\vv\in [0,1]}
\left\{\vv\, 
\ff {\sqrt{2(\ll+K_V^+)}}{\ss{\pi}}\,\exp\left({-\frac{\alpha^2}{2(\ll+K_V^+)}}\right)
+\ss{1-\vv}\,\sqrt{(\ll+K_V^+)}\right\}\|\phi\|_{\infty}^\mstrut.
\end{equation*}
\end{proposition}

\begin{proof} Take 
$t=1/{(\ll+K_V^+)}$ in Lemma \ref{L2.1a}.
\end{proof}

We are now ready to complete the proof of Theorem \ref{T2.0a}.

\begin{proof}[Proof of Theorem \rm\ref{T2.0a}] 
The claims of Theorem \rm\ref{T2.0a} (with the exception of estimate \eqref{T2.0.1_neg}) 
follow directly from the inequalities in Proposition \ref{L2.1b} together with the fact that for any $A,B\ge 0$,
\begin{equation}\label{Eq_AB}
\max_{\vv\in [0,1]}\left\{\vv A+ \ss{1-\vv}B\right\}=B\1_{\{B>2A\}} +\left( A+\ff{B^2}{4A}\right)\1_{\{B\le 2A\}}.
\end{equation}
Finally, to check \eqref{T2.0.1_neg} we may go back to \eqref{Eq:Est} from where we have
$$\|\nabla\phi\|_\infty\le \vv\e^{(\ll +K_V)^+ t/2} \,\|\nabla\phi\|_{\pp D,\infty}^\mstrut
+\ss{1-\vv}\,\e^{\ll t/2} \,\|\phi\|_\infty \,\bigg(\int_0^t |\dot h_s|^2 \e^{K_V s}
\,\d s\bigg)^{1/2}.$$
Taking
$$h_s= \ff{\e^{-K_Vt}-\e^{-K_Vs}}{\e^{-K_Vt}-1},\quad s\in [0,t],$$ we
obtain 
\begin{equation*}
\|\nabla\phi\|_\infty \le \inf_{t>0}\max_{\vv\in [0,1]}\left\{\vv \e^{(\ll+K_V)^+t/ 2}\, \|\nabla\phi\|_{\pp
  D,\infty}^\mstrut+ \|\phi\|_\infty \,\e^{\ll t/2}\ss{1-\vv}\,\left( \ff{K_V
}{1-\e^{-K_Vt}} \right)^{1/2}\right\}.\end{equation*}
Noting that
$$\ff{K_V}{1-\e^{-K_Vt}}\le \ff{K^+_V} {1-\e^{-K_V^+t}}\le t^{-1}\e^{K_V^+t},$$
and taking $t=(K_V^++\ll)^{-1} $
we obtain
\begin{equation*}
\|\nabla\phi\|_\infty \le \ss\e \max_{\vv\in [0,1]} \left\{
 \vv \|\nabla\phi\|_{\pp D,\infty}^\mstrut+  \ss{(1-\vv)(\ll+K^+_V)}
  \,\|\phi\|_\infty \right\}.\end{equation*}
 Applying Lemma \ref{L2.2} and Proposition \ref{P6.3} with $t= 1 /\ll$, we arrive at
\begin{align*} \|\nabla\phi\|_\infty 
&\le \|\phi\|_\infty\max_{\vv\in [0,1]}\left\{\e  \vv\left(\alpha+\ff{\ss{2\ll}}{\ss\pi}\e^{-\ff{\alpha^2 }{2\lambda}} 
+ |\alpha|\land \ff{\alpha^2\ss2}{\ss{\pi\ll}}\right) +\ss{\e(1-\vv)(\ll+K_V^+)} \right\}.\end{align*} 
The proof is then finished as above with observation \eqref{Eq_AB}.
\end{proof}

\section{Proof of Theorem \ref{T1.2}}

As in Section 2, we consider $L=\Delta+\nabla V$ and let $\Eig_N(L)$ be the set of the corresponding non-trivial eigenpairs for the Neumann problem of $L$. We also allow $\pp D=\varnothing$, then we consider the eigenproblem without boundary. We first consider the convex case, then extend to the general situation. In this section, $P_t$ denotes the (Neumann if $\pp D\ne\varnothing$) semigroup generated by $L/2$ on $D$. Let $X_t$ be the corresponding (reflecting) diffusion process which solves the SDE
\begin{equation}\label{SDE2} \d X_t= u_t\circ\d B_t +\frac12\nabla V(X_t)\,\d t +N(X_t)\,\d\ell_t,\end{equation}
where $B_t$ is a $d$-dimensional Euclidean Brownian motion, $u_t$ the horizontal lift of $X_t$ onto the orthonormal frame bundle, and $\ell_t$ the local time of $X_t$ on $\pp D$.

We will apply the following  Bismut type formula for the Neumann semigroup $P_t$, see \cite[Theorem 3.2.1]{W14}, where the multiplicative functional   process $Q_s$ was introduced in \cite{Hsu:2002}.

\begin{theorem}[\cite{W14}]\label{BSMT} Let  $\Ric_D^V\ge -K_V$ and $\I_{\pp D}\ge -\dd$ for some  $K_V\in C(\bar D)$ and $\dd\in C(\pp D)$. Then there exists a $\R^d\otimes\R^d$-valued adapted continuous process $Q_s$ with
\begin{equation}\label{QS} \|Q_t\|\le \exp\left(\frac12\int_0^t K_V(X_s)\d s +\int_0^t \dd(X_s)\d\ell_s\right),\quad s\ge 0,\end{equation}
such that for any $t>0$ and $h\in C^1([0,t])$ with $h(0)=0$, $h(t)=1$, there holds
\begin{equation}\label{BIS} \nabla P_t f = \E\bigg[f(X_t) \int_0^t h'(s) Q_s\d B_s\bigg],\quad f\in \B_b(D).\end{equation}
\end{theorem}

\subsection{The case with convex or empty boundary}
In this part we assume that $\pp D$ is either convex or empty. When $\pp D$ is empty, $D$ is a Riemannian manifold without boundary and $\Eig_N(L)$ denotes the set of eigenpairs for the  eigenproblem without boundary.
In this case, if $\Ric^V\ge K_V$ for some constant $K_V\in\R$, then  $\ll+K_V\ge 0$ for   $(\phi,\ll)\in \Eig_N(L)$,   see for instance \cite{W94}.

\begin{theorem}\label{T3.1} Assume that $\pp D$ is either convex or empty.
\begin{enumerate} \item[$(1)$] If the curvature-dimension condition $\eqref{CD}$ holds, then for any $(\phi,\ll)\in \Eig_N(L)$,
$$  \|\nabla\phi\|_\infty^2 \ge \ff{\ll^2\|\phi\|_\infty^2}{n(\ll+K)}\Big(\ff\ll{\ll+K}\Big)^{\ll/ K} \ge \ff{\ll^2\|\phi\|_\infty^2}{n\e(\ll+K^+)}.$$
\item[$(2)$] If $\Ric^V_D\ge -K_V$ for some constant $K_V\in \R$, then for any $(\phi,\ll)\in \Eig_N(L)$,
$$\ff{\|\nabla \phi\|_\infty ^2}{\|\phi\|_\infty^2} \le \ff{2(\ll+K_V)}\pi \Big(1+\ff{K_V}\ll\Big)^{\ll/{K_V}} \le \ff{2\e(\ll+K_V^+)}\pi.$$ \end{enumerate}  
\end{theorem}
\begin{proof} (a) We start by establishing the lower bound estimate.
By It\^o's formula, for any $(\phi,\ll)\in \Eig_N(L)$ we have
\begin{equation}\label{ITC} \d |\nabla \phi|^2(X_t)= \ff 1 2 L|\nabla \phi|^2(X_t) \,\d t + 2\I_{\pp D}(\nabla \phi,\nabla \phi)(X_t)\,\d\ell_t +\d M_t,\quad t\ge 0,
\end{equation}
where $\ell_t$ is the local time of $X_t$ at $\pp D$, which is an increasing process. Since $\I_{\pp D}\ge 0$, and since \eqref{CD} and $L\phi=-\ll\phi$ imply
$$ \ff 1 2 L|\nabla \phi|^2\ge -(K+\ll)|\nabla \phi|^2+\ff{\ll^2}n \phi^2,$$
 we obtain
 $$\d |\nabla \phi|^2(X_t)\ge \Big(\ff{\ll^2}n \phi^2-(\ll+K)|\nabla \phi|^2\Big)(X_t) \,\d t   +\d M_t,\quad t\ge 0.$$
Noting that for $X_0=x\in D$ we have $$\E[\phi(X_s)^2] \ge (\E[\phi(X_s)])^2= \e^{-\ll s}\phi(x)^2,$$ we arrive at
\begin{align*}\e^{(\ll+K)t} \,\|\nabla\phi\|_\infty^2
&\ge \e^{(\ll+K)t}\,\E[|\nabla\phi|^2(X_t)]\ge \ff{\ll^2}n \int_0^t \e^{(\ll+K)s}\,\E [\phi^2(X_s)]\,\d s \\
&\ge \ff{\ll^2} n \int_0^t \e^{Ks} \phi(x)^2\,\d s = \ff{\ll^2(\e^{Kt}-1)}{nK} \phi(x)^2.\end{align*} 
Multiplying by  $\e^{-(\ll+K)t}$, choosing $t=\ff 1 K\log(1+\ff K \ll)$ (noting that $\ll+K\ge 0$, in case $\ll+K=0$ taking $t\to\infty$),  and  taking  the supremum over  $x\in D$, we finish the proof of (1).\smallskip

(b) Let $\pp D$ be convex and $\Ric_D^V\ge -K_V$ for some constant $K_V$. Then Theorem \ref{BSMT} holds for $\dd=0$, so that
  $$\si_t:=\bigg(\E\int_0^t|h'(s)|^2 \|Q_s\|^2\,\d s\bigg)^{1/2}\le \bigg(\int_0^t |h'(s)|^2\,
\e^{K_V s} \,\d s\bigg)^{1/2}.$$
  Taking $$h(s)=\ff{\int_0^s \e^{-K_V r}\,\d r}{\int_0^t \e^{-K_V r}\,\d r}$$ we obtain
 $$\si_t\le \Big(\ff{K_V}{1-\e^{-K_Vt}}\Big)^{1/2}.$$
 Therefore,
\begin{equation}\label{GRS}\begin{split}  \|\nabla P_tf\|_\infty
&\le \|f\|_\infty \,\E\left|\int_0^th'(s) Q_s\d B_s \right|\\
 &\le \|f\|_\infty \ff{2}{\ss{2\pi}\,\si_t} \int_0^\infty s \exp\left({-\ff{s^2}{2\si_t^2}}\right)\d s\\
& =\|f\|_\infty\ff{\si_t\ss 2}{\ss\pi},\quad t>0,\ f\in \B_b(D).\end{split}\end{equation}
Applying this to $(\phi,\ll)\in \Eig_N(L)$, we obtain
$$\e^{-\ll t/2} |\nabla \phi|\le \|\phi\|_\infty \ff{\si_t\ss 2}{\ss\pi}\le \|\phi\|_\infty \left(\ff{2K_V}{\pi(1-\e^{-2K_Vt})}\right)^{1/2},\quad t>0.$$
Consequently, $\ll+K_V\ge 0$.
Taking $t= \ff 1 {K_V}\log(1+\ff{K_V}\ll)$ as above, we arrive at
\begin{equation*}\ff{\|\nabla \phi\|_\infty^2}{\|\phi\|_\infty^2} \le  \ff{2(\ll+K_V)}\pi \Big(1+ \ff{K_V}\ll\Big)^{\ll/ {K_V}}.
\end{equation*}
\end{proof}

\subsection{The non-convex case}

 When $\pp D$ is non-convex, a conformal change of metric may be performed to make $\pp M$ convex under the new metric; this strategy has been used in \cite{CTh, W07,W10a, W10b} for the study of functional inequalities on non-convex manifolds. According to \cite[Theorem 1.2.5]{W14}, for a strictly positive
 function $f\in C^\infty(\bar D)$ with  
$\I_{\pp D}+N\log f|_{\pp D} \ge 0$,
the boundary $\pp D$ is convex under the metric $f^{-2} \<\cdot,\cdot\>$. For simplicity, we will assume that $f\ge 1$. Hence, we take as class of reference functions
 $$\sD:=\big\{f\in C^2(\bar D)\colon \inf f= 1,
\ \I_{\pp D}+N\log f \ge 0\big\}.$$
 Assume  $\eqref{CD}$ and $\Ric_D^V\ge -K_V$  for some constants $n\ge d$ and $K, K_V\in\R$. For any $f\in \mathscr{D}$ and $\vv\in (0,1)$, define
 \begin{align*} c_\vv(f):= \sup_D\left\{\ff{4\vv|\nabla\log f|^2}{1-\vv} +\vv K+(1-\vv)K_V   -2 L \log f\right\}.\end{align*}
We let $\ll_1^N$ be the smallest non-trivial Neumann eigenvalue of $-L$. The following result implies $\ll_1\ge -c_\vv(f)$.
\begin{theorem}\label{T3.2} Let $f\in \sD$.
\begin{enumerate}[\rm(1)]
\item If   $\eqref{CD}$ and $\Ric_D^V\ge -K_V$  hold for some constants $n\ge d$ and $K, K_V\in\R$. Then for any non-trivial $(\phi,\ll)\in \Eig_N(L)$, we have $\ll+c_\vv(f)\ge 0$ and
$$ \ff{\|f\|_\infty^2\|\nabla\phi\|_\infty^2 }{\|\phi\|_\infty^2}\ge   \sup_{\vv\in (0,1)} \ff{\vv\ll^2}{n(\ll+c_\vv(f))} \Big(\ff\ll{\ll+c_\vv(f)}\Big)^{\ll/ {c_\vv(f)}} \ge  \sup_{\vv\in (0,1)} \ff{\vv\ll^2}{n\e(\ll+c_\vv(f)^+)}.$$
\item Let $\Ric_D^V\ge -K_V$ for some $K_V\in C(\bar D)$, and
$$K(f)=\sup\limits_D\left\{  2 |\nabla\log f|^2+K_V-L\log f\right\}.$$  Then for any non-trivial $(\phi,\ll)\in \Eig_N(L)$, we have $\ll+K(f)\ge 0$ and
$$\ff{\|\nabla \phi\|_\infty^2 }{\|\phi\|_\infty^2\|f\|_\infty^2}    \le  \ff{2(\ll+K(f))} {\pi} \Big(1+\ff{K(f)}\ll\Big)^{\ll/{K(f)}} \le \ff{2\e (\ll+K(f)^+)} {\pi}.$$ \end{enumerate}
\end{theorem}

\begin{proof} Let $f\in \D$ and $(\phi,\ll)\in \Eig_N(L)$.\smallskip

(1) On $\pp D$ we have
\begin{align} N(f^2 |\nabla \phi|^2) &= (N f^2)|\nabla \phi|^2 + f^2N|\nabla \phi|^2\notag\\
& =f^2\big((N\log f^2) |\nabla \phi|^2 + 2\I_{\pp D}(\nabla \phi,\nabla\phi)\big)\notag\\
& =2f^2\big((N\log f) |\nabla \phi|^2 + \I_{\pp D}(\nabla \phi,\nabla\phi)\big)\ge 0.
\label{BN}\end{align}
Next, by the Bochner-Weitzenb\"ock formula,
using that $\Ric_D^V\ge -K_V$ and $L\phi=-\ll\phi$,
we observe
\begin{align*}\ff  1 2 L |\nabla \phi|^2 &=\ff 1 2 L |\nabla\phi|^2 -\<\nabla L\phi,\nabla\phi\>
-\ll |\nabla\phi|^2\\
&\ge \|\Hess_\phi\|_{\HS}^2 -(K_V+\ll)|\nabla\phi|^2.
\end{align*}
Combining this with \eqref{LMN}, for any $\vv\in (0,1)$,  we obtain
\begin{align*} \ff{f^2} 2 &L |\nabla \phi|^2 +\<\nabla f^2, \nabla|\nabla \phi|^2\> \\
&\ge -f^2 (\vv K +(1-\vv) K_V +\ll) |\nabla\phi|^2 +\ff{\vv \ll^2}n f^2 \phi^2\\
&\quad +(1-\vv) f^2\|\Hess_\phi\|_{\HS}^2 - 2 \|\Hess_\phi\|_{\HS}^\mstrut\times|\nabla f^2|\times |\nabla \phi|\\
&\ge -\left\{\ff{|\nabla\log f^2|^2}{1-\vv} +\vv K+ (1-\vv) K_V +\ll \right\} f^2|\nabla\phi|^2 + \ff{\vv \ll^2}n f^2\phi^2.\end{align*}
Combining this with \eqref{BN} and applying It\^o's formula, we obtain
\begin{align*} \d &(f^2|\nabla\phi|^2)(X_t) \mequal \frac12L(f^2|\nabla\phi|^2)(X_t)\,\d t+ N(f^2|\nabla \phi|^2)(X_t)\,\d\ell_t \\
&\ge -\frac12\Big(f^2L|\nabla\phi|^2 +2\<\nabla f^2, \nabla|\nabla \phi|^2\> +|\nabla\phi|^2 Lf^2\Big)(X_t)\,\d t \\
&\ge \left\{ \ff{\vv \ll^2}{n}f^2 \phi^2 -\left(\ff{|\nabla\log f^2|^2}{1-\vv} +\vv K+(1-\vv)K_V +\ll-f^{-2}Lf^2\right)f^2|\nabla\phi|^2\right\}(X_t)\,\d t\\
&\ge \left(\ff{\vv \ll^2} n \phi^2 - \big(\ll+c_\vv(f)\big)f^2|\nabla \phi|^2\right)(X_t)\,\d t.\end{align*}
Hence, for $X_0=x\in D$,
\begin{align*}  \|f\|_\infty^2 \,\|\nabla \phi\|_\infty^2\,\e^{(\ll+ c_\vv(f)) t}
&\ge \E\left[\e^{c_\vv(f)t}(f^2|\nabla\phi|^2)(X_t)\right]\\
&\ge \ff{\vv \ll^2}n \int_0^t \e^{(\ll+c_\vv(f))s}\, \E[\phi(X_s)^2]\,\d s\\
&\ge \ff{\vv \ll^2}n \int_0^t \e^{c_\vv(f) s}  \phi(x)^2\,\d s\\
&= \ff{\vv\ll^2(\e^{c_\vv(f)t}-1)}{nc_\vv(f)} \,\phi(x)^2.\end{align*}
This implies $\ll+c_\vv(f)\ge 0$ and
\begin{align*} &\ff{\|f\|_\infty^2\|\nabla\phi\|_\infty^2}{ \|\phi\|_\infty^2} \ge \sup_{t>0} \ff{\vv\ll^2\left(\e^{-\ll t}-\e^{-(\ll+c_\vv(f))t}\right)}{nc_\vv(f)}\\
&=
\ff{\vv\ll^2}{n(\ll+c_\vv(f))} \Big(\ff \ll {\ll +c_\vv(f)}\Big)^{\ll/{c_\vv(f)}} \ge \ff{\vv\ll^2}{n\e(\ll+c_\vv(f)^+)}.\end{align*}

(2) The claim could be derived from \cite[inequality (2.12)]{CTh}.
For the sake of completeness we include a sketch of the proof.
For any $p>1$, let
$$K_p(f) =  \sup_D\left\{K_V + p|\nabla\log f|^2-L\log f\right\}.$$
Note that $p|\nabla\log f|^2-L\log f=p^{-1}f^pLf^{-p}$.
Since $f\in \sD$ implies $\I_{\pp D}\geq-N\log f$, we have
\begin{align*}\|Q_t\|^2&\leq
\exp\left(\int_0^t K_V(X_s)\,\d s + 2\int_0^t N\log f (X_s)\,\d\ell_s\right)\\
&\leq\exp\big(K_p(f)t\big)\exp\left(-\frac1{p}\int_0^t(f^pLf^{-p})(X_s)\,\d s
+2\int_0^t N\log f (X_s)\,\d\ell_s\right).
\end{align*}
As
\begin{align*}
\d f^{-p}(X_t)&\mequal\frac12Lf^{-p}(X_t)\,\d t+Nf^{-p}(X_t)\,\d\ell_t\\
&=-f^{-p}(X_t)\left(-\frac12f^{p}Lf^{-p}(X_t)\,\d t+pN\log f (X_t)\,\d\ell_t\right),
\end{align*}
we obtain that
$$M_t:=f^{-p}(X_t)\exp\left(-\frac12\int_0^tf^{p}(X_s)Lf^{-p}(X_s)\,\d s
+p\int_0^tN\log f (X_s)\,\d\ell_s \right)$$
is a (local) martingale. Proceeding as in the proof of \cite[Corollary 3.2.8]{W14}
or \cite[Theorem~2.4]{CTh}, we get
\begin{align*}
\|f\|^{-p}_\infty&\,\E\left[\exp\left(-\frac12\int_0^tf^{p}(X_s)Lf^{-p}(X_s)\,\d s+p\int_0^tN\log f (X_s)\,\d\ell_s \right)\right]\\
&\leq
\E\left[f^{-p}(X_t)\exp\left(-\frac12\int_0^tf^{p}(X_s)Lf^{-p}(X_s)\,\d s+p\int_0^tN\log f (X_s)\,\d\ell_s \right)\right]\\
&= f^{-p}(x)\leq1,
\end{align*}
since $f\geq1$ by assumption. This shows that
\begin{align*}\E\|Q_t\|^2\leq \e^{K_p(f)t}\,\|f\|^{p}_\infty,\quad t\geq0.
\end{align*}
Combining this for $p=2$ with  Theorem \ref{BSMT} and denoting $K(f)=K_2(f)$, we obtain
$$\si_t^2:= \E \int_0^t |h'(s)|^2 \|Q_s\|^2\,\d s \le \|f\|_\infty^2 \int_0^t |h'(s)|^2\e^{K(f)s}\,\d s.$$
Therefore, repeating step (b) in the proof of Theorem \ref{T3.1} with $K(f)$ replacing $K_V$, we finish the proof of (2).
\end{proof}



\begin{thebibliography}{10}

\bibitem{Arnaudon-Thalmaier-Wang:09}
Arnaudon, M.,  A.~Thalmaier and F.-Y.~Wang.
``Gradient estimates and Harnack inequalities on non-compact Riemannian manifolds.''
\textit{Stochastic Process. Appl.} {119}, no.~10 (2009): 3653--3670.

\bibitem{CTh} Cheng,~L.-Y., J.~Thompson and A.~Thalmaier. 
``Functional inequalities on manifolds with non-convex boundary.'' 
\textit{Sci. China Math.} {61}, no.~8 (2018): 1421-1436.

\bibitem{Hsu} Hsu,~E.\,P. \textit{Stochastic Analysis on Manifolds}.  
Graduate Studies in Mathematics, vol.~38. American Mathematical Society, 2002.

\bibitem{Hsu:2002}
Hsu, E.\,P. ``Multiplicative functional for the heat equation on
  manifolds with boundary.'' \textit{Michigan Math.~J.} {50}, no.~2 (2002): 351--367.

\bibitem{HSX} Hu, J., Y.~Shi and B.~Xu. 
``The gradient estimate of a Neumann eigenfunction on a compact manifold with boundary.'' 
\textit{Chin. Ann. Math. Ser. B} {36}, no.~6 (2015): 991--1000.


\bibitem{Kendall} Kendall, W.\,S. 
``The radial part of Brownian motion on a manifold: a semimartingale property.'' 
\textit{Ann. Probab.} {15}, no.~4 (1987): 1491--1500.


\bibitem{G1} Shi, Y. and B.~Xin, 
``Gradient estimate of a Dirichlet eigenfunction on a compact manifold with boundary.'' 
\textit{Forum Math.} {25}, no.~2 (2013): 229--240.

\bibitem{W94} Wang, F.-Y.
``Application of coupling method to the Neumann eigenvalue problem.'' 
\textit{Probab. Theory Related Fields}  
{98}, no.~8 (1994): 299--306.

\bibitem{W95} Wang, F.-Y. ``Estimates of the first Dirichlet eigenvalue by using diffusion processes.''
\textit{Probab. Theory Related Fields} {101}, no.~3 (1995): 363--369.

\bibitem{W04}
Wang, F.-Y. ``Gradient estimates of Dirichlet semigroups and applications to isoperimetric inequalities.'' 
\textit{Ann. Probab.} {32}, no.~1A (2004): 424--440.

\bibitem{W05} Wang, F.-Y. ``Gradient estimates and the first
Neumann eigenvalue on manifolds with boundary.''  
\textit{Stochastic Process. Appl.} {115} (2005),  no.~9, 1475--1486.

\bibitem{W07} Wang, F.-Y. ``Estimates of the first Neumann eigenvalue and the log-Sobolev constant on non-convex manifolds.'' 
\textit{Math. Nachr.} {280}, no.~12 (2007): 1431--1439.

\bibitem{W10a} Wang, F.-Y. ``Harnack inequalities on manifolds with boundary and applications.'' 
\textit{J. Math. Pures Appl.} {94}, no.~3 (2010): 304--321.
 	
\bibitem{W10b} Wang, F.-Y. ``Semigroup properties for the second fundamental form.''  
Doc. Math. {15} (2010): 543--559.

\bibitem{W14} Wang, F.-Y. \textit{Analysis for Diffusion Processes on Riemannian Manifolds.} 
Advanced Series on Statistical Science \& Applied Probability, vol.~18. 
World Scientific Publishing Co.~Pte.~Ltd., 2014.


\end{thebibliography}
\end{document}